\documentclass[12pt]{article}

\usepackage[margin=1in]{geometry}  
\usepackage{graphicx}              
\usepackage{amsmath}               
\usepackage{amsfonts}              
\usepackage{amsthm}                

\newtheorem{thm}{Theorem}[section]
\newtheorem{lem}[thm]{Lemma}
\newtheorem{prop}[thm]{Proposition}

\begin{document}

\title{On the gradient flows on Finsler manifolds}
\author{N. Shojaee, M. M. Rezaii\\
 Department of  Mathematics and Computer science\\ 
 Amirkabir University of Technology (Polytechnic of Tehran),\\ 
Tehran, Iran}

\maketitle

\begin{abstract}
The purpose of this article is to provide a general overview of curvature functional in Finsler geometry and use its information to introduce the gradient flow on Finsler manifolds. For this purpose, we first prove that the space of Finslerian metrics is a Riemannian manifold. Then it is given a decomposition for the tangent space of this Riemannian manifold by means of Riemannian metric and the Berger-Ebin theorem. Next, Finslerian functional is introduced and show that Akbar-Zadeh curvature functional is the example of Finslerian functional. After that, the critical points of Akbar-Zadeh functional are found in two situations. Based on the constant Indicatrix volume and restricting its variations to the point-wise conformal transformations, we prove that the critical points of functional are metrics of constant Ricci-directional curvature. Finally, the gradient flow of Akbar-Zadeh curvature functional in this special direction is introduced as a good candidate for evolving to the metric with constant second type scalar curvature and we compare this new flow with introducing Ricci flow in Finsler geometry. \\


Subjclass[2000]: {53B40, 58B20, 58E11}

\end{abstract}

\newtheorem{mythm}{Theorem}[section]
\newtheorem{mylem}[thm]{Lemma}
\newtheorem{myprop}[thm]{Proposition}
\newtheorem{mycor}[thm]{Corollary}
\newtheorem{myconj}[thm]{Conjecture}
\theoremstyle{definition} 
\newtheorem{definition}{Definition}[section]
\theoremstyle{remark} 
\newtheorem{remark}{Remark}[section] 
\newtheorem{example}{Example}

\section{Introduction}
Nonlinear heat flows were first appeared in Riemannian geometry in 1964, when Eells and Sampson have introduced the harmonic map heat flow as the gradient flow of the energy functional $E(u)=\int_{M}|\nabla u|^2dV$ \cite{ES}. They have used this flow to deform given maps $u:M\rightarrow N$ between two manifolds into the extremal maps which are critical points of the energy functional $E(u)$ in the sense of the calculus of variation.\\
One of the fundamental problems in differential geometry is to find canonical metrics on Riemannian manifolds, that is, the metrics which are highly symmetrical, like those of constant curvature in some senses. It was first Hamilton, who used the idea of evolving an object to such an ideal state by a nonlinear heat flow and invented the Ricci flow in 1981, cf. \cite{HA}. He has proved that a Riemannian metric of strictly positive Ricci curvature on a compact 3-manifold can be deformed into a metric of positive constant curvature. Similar results for a compact $n$-manifold have been obtained by G. Huisken \cite{HU}, C. Margerin \cite{MA} and S. Nishikawa \cite{NI}.\\
The stationary metrics under the Ricci flow are Ricci flat metrics which are also the critical points of the Einstein-Hilbert functional ${\mathcal{E}}(g)=\int_MRdV$ but the Ricci flow is not exactly the gradient flow of this functional. It is just the first part of the Einstein-Hilbert functional's gradient flow, $\partial_tg_{ij}=-R_{ij}+(R/n) g_{ij}$.  If this functional is restricted to the class of conformal metrics, then it has a strictly parabolic gradient flow called the Yamabe flow. Hamilton proved that there is not any functional such that its gradient flow is exactly the Ricci flow. Perelman recently improved the Einstien-Hilbert functional and introduced ${\cal F}$-entropy functional ${\cal F}(g_{ij},f)=\int_M(|\nabla f|^2+R)e^{-f}dV$. This functional had a system of PDEs such that one of these equations was the Ricci flow \cite{M}.\\
Moreover, the gradient flows have fruitful roles in global analysis on manifolds and in different branches of applied sciences such as image processing and biological problems. Finsler geometry is a natural generalization of Riemannian geometry. Therefore, it is natural to extend gradient flow on Finsler manifolds. For the first attempt in this topic, we can mention the concept of the Ricci flow introduced by Bao \cite{B1}. He introduced $\partial_t\log F=-Ric$ as a scalar form of the Ricci flow on Finsler manifolds since it guarantees that the metric is derived from Finsler structure without needing to check integrability conditions in each step of work. Recently, Ohta and Sturm studied the heat flows on Finsler manifolds in \cite{OS}.\\
In the present work, the space of all Finsler metrics on a given manifold is studied. It is proved that this space has a Riemannian manifold structure and is represented by $\cal{M}_F$ (see Theorem (3.2)). This manifold structure is important to make sure that the solutions of heat flows come from a Finsler structure. A heat flow is placed in the tangent space of manifold ${\cal M}_F$ and if it has some solutions, then they are located in this manifold according to the concept of integral curves. So it gurantees the solutions of heat flows come from a Finsler metric no need to use the scalar form of flows.  In the forth section, by means of the metric structure on manifold ${\cal M}_F$ and the Berger-Ebin theorem, two different splits are given for the tangent space of ${\cal M_F}$ (see Theorems 4.2, 4.3), which are the natural extension of the Riemannian case, see \cite{E1, BE2, F1, S}. In this way, in the fifth section, the variation of Finsler metrics is defined a bit more complete than what is mentioned by Akbar-Zadeh in \cite{A2} (see Definition 5.2). Furthermore, the Berger-Ebin decomposition of the tangent space ${\cal M}_F$  helps us to calculate the variation of curvature functionals in different directions and find another heat flow $\partial_tg_{ij}=-H(u,u)g_{ij}$ which is a weakly parabolic equation (see Corallary 5.4). Furthermore, we define Finslerian functionals (see Definition 5.3) and we compute variations of Akbar-Zadeh functional in the point-wise conformal transformation direction and prove that the critical points of this functional are of constant Ricci-directional curvature (see Theorem 5.3).
\section{Preliminaries}
Let $(M,g)$ be a connected, compact Finsler manifold. That is, there is a function $F$ on the tangent bundle $TM$ satisfying the following conditions:
\begin{itemize}
\item $F$ is a smooth function on the entire slit tangent bundle $TM_o$.
\item $F$ is a positive homogeneous function on the second variable, $y$.
\item The matrix $(g_{ij})$, $g_{ij}(x,y)=\frac{1}{2}\frac{\partial^2F^2}{\partial y^i\partial y^j}$ is non-degenerate.
\end{itemize}
\subsection{Connections on Finsler manifold}
Geodesics of a Finsler structure $F$ are characterized locally by $\frac{d^2x^i}{dt^2}+2G^i(x,\frac{dx}{dt})=0$, where  
$G^i=\frac{1}{4}g^{ih}(\frac{\partial^2 F^2}{\partial y^h\partial x^j}y^j-\frac{\partial F^2}{\partial x^h})$ are called geodesic spray coefficients.
Let $G^i_j=\frac{\partial G^i}{\partial y^j}$ be the coefficients of a nonlinear connection on $TM$. By means of this nonlinear connection, the tangent space $TM_o$ splits into horizontal and vertical subspaces. $TTM_0$ spanned by $\{\frac{\delta}{\delta x^i},\frac{\partial}{\partial y^i}\}$, where $\frac{\delta}{\delta x^i}:=\frac{\partial}{\partial x^i}-G_i^j\frac{\partial}{\partial y^j}$ are called Berwald bases and their dual bases are denoted by $\{dx^i,\delta y^i\}$, where $\delta y^i:=dy^i+G_j^idx^j$. Furthermore, this nonlinear connection can be used to define a linear connection called the Berwald connection and its connection $1$-forms are defined locally by $\pi^i_j=G^i_{jk}dx^k$ where $G^i_{jk}=\frac{\partial G^i_j}{\partial y^k}$.
The connection $1$-forms of the Cartan connection are defined by $\tilde{\nabla}\frac{\partial}{\partial x^i}=\omega_i^j\frac{\partial}{\partial x^j}$, where $\omega_j^i=\Gamma^i_{jk}dx^k+C^i_{jk}\delta y^k$ such that
\begin{equation*}
\Gamma_{jk}^i=\frac{1}{2}g^{im}(\frac{\partial g_{mj}}{\partial x^k}+\frac{\partial g_{mk}}{\partial x^j}-\frac{\partial g_{kj}}{\partial x^m})-(C^i_{js}G^s_k+C^i_{ks}G^s_j-C_{kjs}G^{si}),
\end{equation*}
 and 
\begin{align}\label{eq11}
C^i_{jk}=\frac{1}{2}g^{im}(\frac{\partial g_{mj}}{\partial y^k}+\frac{\partial g_{mk}}{\partial y^j}-\frac{\partial g_{kj}}{\partial y^m}),
\end{align}
Hence we have $\tilde{\nabla}=\nabla+\dot{\nabla}$ where, $\nabla$ is the horizontal coeffiecients of the Cartan connection and $\dot{\nabla}$ is the vertical coeffiecients of the Finslerian(Cartan) connection. For more details of this subsection see \cite{BCS}.\\
\subsection{The curvature tensors of Finsler manifold}
The hh-curvature of the Cartan and Berwald connections are denoted respectively by $R^{~i}_{j~kl}$ and $H^{~i}_{j~kl}$. They are related by \cite{A2},
$$R^{~i}_{j~kl}= H^{~i}_{j~kl}+C^{i}_{jr}R^{~r}_{o~kl}+\nabla_l\nabla_oC^{i}_{jk}-\nabla_k\nabla_o C^{i}_{jl}+\nabla_oC^{i}_{lr}\nabla_oC^{r}_{jk}-\nabla_oC^i_{kr}\nabla_oC^{r}_{jl},$$
The Ricci tensor is defined in different ways in Finsler geometry. In the present work, we consider the Akbar-Zadeh's definition of Ricci tensor given by ${\tilde H}_{ij}=1/2\frac{\partial^2}{\partial y^i\partial y^j}(H_{rs}y^ry^s)$, where $H_{ij}=g^{ks}H_{ikjs}$. The Ricci-directional curvature is defined by $H(u,u)=g^{ik}H_{ijkl}u^ju^l$ where, $u^i=\frac{y^i}{F}$. The scalar curvature of second type is defined by $\tilde{H}=g^{ij}\tilde{H}_{ij}$. According to the above formula, it is not important which connections are used for defining the Ricci tensor and the Ricci-directional curvature.
\subsection{The Indicatrix bundle}
Let $x_0\in M$, define $S_{x_0}M=\{y\in T_{x_0}M|F(x_0, y)=1\}$ which is called Indicatrix at the point $x_0$. Put $SM:=\mathop{\cup}\limits_{x_0\in M} S_{x_0}M$, $SM$ is called the
Indicatrix bundle of a Finsler structure $F$. All the geometric objects on $SM$ are positive homogeneous of degree zero. If $f$ is a function on $SM$ and $\partial_\alpha$'s $(1\leq\alpha\leq n-1)$ are partial derivatives along Indicatrix then its derivative is $\partial_{\alpha} f=v^i_\alpha\frac{\partial f}{\partial y^i}$ where $\partial_\alpha=v^i_\alpha\frac{\partial }{\partial y^i}$ and $v^i_\alpha$ is the transition matrix of rank $(n-1)$. So the coefficients of the induced metric on $SM$ are $g_{\alpha\gamma}=v^i_\alpha v^j_\beta g_{ij}$ and since the vertical Liouville vector field $L$ is normal to the Indicatrix with respect to this metric, we have 
\begin{equation*}
(g_{ij})_{1\leq i,j\leq n}=
\begin{bmatrix}
(g_{\alpha \gamma})& 0\\
0 & 1 
\end{bmatrix}.
\end{equation*}
The dual bases of $\partial_\alpha$ is denoted by $\beta^\alpha$ and is defined by $\beta^\alpha=\nabla u=\omega^\alpha_n$, cf. \cite{A3}.\\
The Indicatrix bundle $SM$ is always orientable and the compactness of $M$ provided that $SM$ is compact, too. These two properties of $SM$ permit us to define integral on Finsler manifolds and a global inner 
product on $SM$. The volume element of the Indicatrix bundle is denoted by $(2n-1)$-form $\eta$, cf. \cite{A2},
\begin{align}\label{eqn 22}
\eta:=\frac{(-1)^N}{(n-1)}\phi,\quad \phi=\omega\wedge(d\omega)^{(n-1)},\quad N=\frac{n(n-1)}{2}.
\end{align}
where $\omega$ is the Hilbert form. On the tensor spaces on $SM$, The canonical (point-wise) scalar product is denoted by $<.|.>$ and the global scalar product on their sections is denoted by $(.|.)=\int_{SM}<.|.>\eta$.
The codifferential operator on the space of differentiable $1$-forms is defined on $SM$ by, cf. \cite{A2},
\begin{equation}
\delta a=-(\nabla^ja_j-a_j\nabla_0C^j),
\end{equation}
where, $a$ is a horizontal 1-form on $SM$. And
\begin{equation}
\delta b=-F(\dot{\nabla}_jb^j+b_jC^j)=-Fg^{ij}\partial_jb_i.
\end{equation}
where, $b$ is a vertical 1-form on $SM$. 
\subsection{Curvature functional}
Akbar-Zadeh defined different functionals by means of different curvature tensors in \cite{A2,A1}. The more general case among them is
\begin{align}\label{eqn 19}
I(g_t)=\int_{SM}\hat{H}_t\eta_t,
\end{align}
where $\hat{H}=\tilde{H}-c(x)H(u,u)$. The critical points of this functional are called generalized Einstein metrics. More preciesly, we have the following:
\begin{definition}
A Finslerian manifold is called a generalized Einstein manifold (GEM) if the Ricci-directional curvature is independent of the direction. That is to say
$$\tilde{H}_{ij}(x,y)=C(x)g_{ij}(x,y).$$
\end{definition}
Through finding critical points of the functional $I(g_t)$, Akbar-Zadeh proved that 
\begin{align}\label{eqn 20}
C(x)=nH(u,u)=\tilde{H}.
\end{align}
 So the Ricci-directional curvature is related to the second type scalar curvature, see \cite{A2} for more details.
\section{The space of Finsler metrics}
The space of Riemannian metrics on a given manifold is an infinite dimensional manifold. It is easy to see this property since the Riemannian metrics space is the open and convex set of the space of all sections of $S^2T^*M$. Ebin used the manifold structure in \cite{E1} and gave a Riemannian structure to the manifold of Riemannian metrics on a compact manifold $M$. The aim of this section is to consider the geometry of the space of Finslerian metrics. Dealing with Finslerian case is not as easy as Riemannian case because of PDEs and integrability conditions for defining the Finsler metrics. The outline of the proof is to start by the generalized Lagrange metrics and restricted it to find a suitable PDE for introducing Finsler metric space.
The generalized Lagrange metric is a metric structure on $\pi^*TM$ or $VTM$ and is defined as follows:
\begin{definition}
A generalized Lagrange metric, briefly a GL-metric on an $n$-dimensional manifold $M$, is a $(0,2)$ d-type tensor field $g_{ij}(x,y)$ on $TM$ satisfying the following
\begin{itemize}
\item $g_{ij}(x,y)=g_{ji}(x,y)$, i.e. it is symmetric,
\item det$g_{ij}(x,y)\neq 0$, i.e. it is regular,
\item The quadratic form $g_{ij}(x,y)\xi^i\xi^j,\xi\in\mathbb R^n$ has a constant signature.
\end{itemize}
\end{definition}
If we only consider positive signature, then $g(x,y)$ is a Euclidean product of the vector space $\pi^*|_zTM$ for each $z=(x,y)\in U\subset TM$. So $\pi^*TM$ is a Riemann vector bundle over $TM$. A GL-metric is called a Lagrange metric, if there is a potential function $L:TM\rightarrow {\mathbb R}$ such that 
\begin{align}\label{eqn 21}
g_{ij}(x,y)=\frac{1}{2}\frac{\partial^2 L}{\partial y^i \partial y^j}(x,y),
\end{align}
are components of  a positive definite matrix. A GL-metric is reducible to a Lagrange metric if and only if the Cartan tensor (\ref{eqn11}) is symmetric in all three indices. This condition is equivalent to the integrability condition of the system (\ref{eqn 21}) i.e. $\frac{\partial g_{ij}}{\partial y^k}=\frac{\partial g_{ik}}{\partial y^j}$ is satisfied. It signifies that the equation (\ref{eq11}) is reduced to the form $C_{ijk}=\frac{1}{2}\frac{\partial g_{ij}}{\partial y^k}=\frac{1}{2}\partial_k g_{ij}$. Furthermore, the coefficients of a Finslerian metric are zero homogeneous, so they are lying on $SM$. Hence a Lagrange metric is reduced to a Finsler metric if and only if the coefficients of the metric are satisfied with a system of the linear partial differential equations, $y^k\frac{\partial g_{ij}}{\partial y^k}=0$, see \cite{BM} for more details. So the problem of introducing the space of Finsler metrics is reduced to finding the solution space of the following system:
\begin{equation}\label{eqn14}
\left\{ \begin{array}{ll}
y^i\partial_i g_{jk}=0; &  i,j,k=1,\dots ,n\\
g(\xi,\xi)> 0; & \xi\in\Gamma(\pi^*TM_0)
\end{array} \right.
\end{equation}
We note that since these equations are defined in L-metrics space so the potential function is always defined by 
$$L(x,y)=g_{ij}(x,y)y^iy^j.$$
for the solutions of (\ref{eqn14}). It means that the integrability condition is satisfied for these solutions. Now, the procedure is to define another system of equations which is equivalent to (\ref{eqn14}). 
\begin{definition}
Let $E$ and $F$ be vector bundles over the manifold $M$. A linear differential operator of order $q$ from $E$ to $F$ is a map $\phi o j^q:E\rightarrow F$ between the sets of germs of sections $E$ and $F$ where, $\phi:J^q(E)\rightarrow F$ is a vector bundle morphism and $J^q(E)$ is the jet bundle of $E$ of order $q$.
\end{definition}
A GL-metric is a field of cones on $S^2\pi^*T^*M$, that is
\begin{align}
k:TM\rightarrow S^2\pi^*T^*M\\
z\rightarrow k(z)\subset E_z\nonumber
\end{align}
where $k(z)=\{g_{ij}\in S^2\pi^*T^*M| det g_{ij}>0\}\cup\{g_{ij}\in S^2\pi^*T^*M| det g_{ij}<0\}$. So the space of GL-metrics is a symmetric $2$-forms bundle over $TM$ endowed with a field of cones which is denoted by $E:=[S^2\pi^*T^*M;K]$ cf. \cite{BE}.
Let $F$ be the subbundle of $J^1E$ which is spanned at each point $z\in TM$ by $(u^{ij},u_k^{ij},u_{\alpha}^{ij},u^{ij}_L)$ where, $L=y^k\frac{\partial}{\partial y^k}$ is the vertical Liouville vector field.
Suppose that $P:\Gamma(E)\rightarrow \Gamma(F)$ is a linear first order differential operator which is defined by $P(g):=\Phi o j^1(g)=y^k\partial_kg_{ij}$ c.f \cite{G}.
\begin{definition}
A morphism of vector bundles $\sigma(P):S^qT^*M\otimes E\rightarrow F$ which is fibered over $P:E\rightarrow F$ is called the symbol of $P$.
\end{definition}
The symbol of $P$ is defined by: 
\begin{align*}
\sigma(P):T^*(TM)\otimes E\rightarrow F\\
\sigma_t(P)=P(fg),
\end{align*}
where $t=df$. In local coordinate, we have $P(fg)=y^k\partial_k(fg_{ij})$. So by means of the integrability condition for system (\ref{eqn14}), the kernel of this symbol is the space of conformal Finsler metrics. For any $s\geq 3$, the vector space $$V_s:=(T^*TM\otimes E)\cap(S^{s-1}T^*TM\otimes ker(\sigma(P))),$$ 
is vanish. Therefore, the system $(P,E,F)$ is of finite type. So the equation $P(g)=0$ is equivalent to the closed system of PDEs of the form $\partial_k g_{ij}=\psi_k(ij)$ where, $\psi_k(ij)$ are a combination of the homogeneous functions of order $-1$ of $y^i$, $L(x,y)$, $\frac{\partial L}{\partial y^i}(x,y)$ and $\frac{\partial^2 L}{\partial y^i\partial y^j}(x,y)$.  Hence the system of equations (\ref{eqn14}) is equivalent to the following system:
\begin{equation}\label{eqn 15}
\left\{ \begin{array}{ll}
\partial_k g_{ij}=\psi_k(ij)\quad ; &  i,j,k=1,\dots ,n‎\\
g(\xi,\xi)> 0\quad\quad\quad\quad ; & \xi\in\Gamma(\pi^*TM_0).
\end{array} \right.
\end{equation}
It will thus be sufficient to prove that the system (\ref{eqn 15}) has a solution,  see \cite{Du}. Since this system is of finite type, i.e. the higher order derivatives can be written in lower order derivatives, the integrability condition is always true for this system.
\begin{prop}
The system of PDEs (\ref{eqn 15}) has a solution.
\end{prop}
\begin{proof}
The $1$-forms associated with this linear system are $d^vg_{ij}-\psi_k(ij)dy^k=0$, so the annihilator of these $1$-forms are $X_k=\partial_k+\psi_k(ij)\frac{\partial}{\partial g_{ij}}$. According to the Frobenius theorem, this system has a solution if and only if $[X_k,X_l]=0$.
This condition is equivalent: 
\begin{equation}
\partial_k\psi_l(ij)-\partial_l\psi_k(ij)+\psi_l(mn)\frac{\partial \psi_k(ij)}{\partial g_{mn}}-\psi_k(mn)\frac{\partial \psi_l(ij)}{\partial g_{mn}}=0.
\end{equation}
By integrability condition, we have $\partial_k\psi_l(ij)=\partial_l\psi_k(ij)$, so this equation is reduced to
\begin{equation}\label{eqn 18}
\psi_l(mn)\frac{\partial \psi_k(ij)}{\partial g_{mn}}-\psi_k(mn)\frac{\partial \psi_l(ij)}{\partial g_{mn}}=0.
\end{equation}
Represent the set of algeraic equations (\ref{eqn 18}) by $F_K(z,g)=0$, where $K=1,\dots ,n^2$. So we have maximum $n^2$ independent and it yields that the system of equations (\ref{eqn 15}) has a solution according to Theorem (2.1) of \cite{Du}.
\end{proof}
\begin{thm}
The space of all Finsler metrics on a compact manifold $M$ is a Riemannian manifold.
\end{thm}
\begin{proof}
Let $F$ be a solution of  (\ref{eqn 18}) so $F$ is a homogeneous function of order $2$ on $TM$. It means that the solution space of (\ref{eqn 18}) is an infinite dimensional manifold. Suppose $g$ is a solution of (\ref{eqn 15}), so $g\in \Gamma(E)$ such that it is zero homogeneous and satisfies in the integrability condition i.e. $g_{ij}=1/2\frac{\partial^2F}{\partial y^i\partial y^j}$. So the solution space of (\ref{eqn 15}) is an infinite dimensional manifold, too. This solution space is represented by ${\cal M}_F$. 
For every $g\in {\cal M}_F$, the tangent space of this manifold is the space of all symmetric $2$-forms which are positive homogeneous of degree zero and symmetric in all three indices i.e.
$$T_g {\cal M}_F=\{h\in S^2(\pi^*_sT^*M)| \partial_jh_{ik}=\partial_kh_{ij}\},$$
Define the global inner product on ${\cal M}_F$ by 
\begin{align}\label{eqn10}
(a,b)_g:=\int_{SM}<a,b>\eta,
\end{align}
where $a,b\in T_g {\cal M}_F$. The local inner product is defined by $<a,b>:=g^{-1}ag^{-1}b$ and we suppose that $a$ and $b$ are square integrable. This inner product smoothly depends on $g$. Therefore, the pair $({\cal M}_F, (.|.))$ is an infinite dimensional Riemannian manifold.
\end{proof}
\section{Different decompositions of the tangent space of ${\cal M}_F$}
It is well known that $\pi^*TM$ is isomorphic to $VTM$. Let us consider a section $s:M\rightarrow TM$. The pullback bundle $s^*VTM$ is a vector bundle over $M$ and for all $x\in M$ there is an isomorphism $\Pi_x:(VTM)_{s(x)}\rightarrow(s^*VTM)_x\cong (s^*\pi^*TM)_x$. We use this isomorphism frequently without notification in this work
. Consider a vector field $V\in\Gamma(TM)$ and denote by $\eta_t$ the $1$-parameter local flow of $V$. 
Let $\tilde{\eta}$ be the natural extension of $\eta$ on $TM$ defined by $\tilde{\eta}_t:(x^i,y^i)\rightarrow (x^i+tv^i, y^i+ty^m\frac{\partial v^i}{\partial x^m})$.
Clearly, $\hat{V}:=\frac{d}{dt}|_{t=0}\tilde{\eta}_t$ is the complete lift of the vector field $V$ on $TM$.\\
Let $X=X^i\frac{\partial}{\partial x^i}$ be a section of $\pi_s^*TM$. Consider the canonical linear mapping $\varrho:T_zTM\rightarrow \pi_s^*T_xM$ which is defined by $\varrho_z(\frac{\delta}{\delta x^i})=\frac{\partial}{\partial x^i}|_x$ and $\varrho_z(\frac{\partial}{\partial y^i})=0$ in local coordinates. Suppose $\hat X,\hat Y$ and $\hat Z$ are sections of $TTM$ so by using the Lie derivative and torsion definitions and the properties of Cartan connection we obtain:
\begin{align*}
L_{\hat X}g(\varrho\hat Y,\varrho\hat Z)&=L_{\hat X}g(Y,Z)\\
&=g(symm(\nabla X)\hat Y,Z)+g(Y,symm(\dot{\nabla} X)\hat Z)\\
&+2g(T(X,\dot{Z}),Y)+g(T(\dot{X},Z),Y)+g(T(\dot X,Y),Z),
\end{align*}
where $\dot{X}:=\nabla_{\hat{X}}L$ and $g(symm(\nabla X)\hat Y,Z):=g(\nabla_{H\hat Y} X,Z)+g(Y,\nabla_{H\hat Z} X)$. It is defined similarly for vertical connection.\\
Now, Let $\hat X$ be the complete lift of a vector field $X$ on $M$. Replacing this vector field in the above Lie derivative equation, and using $y^m\frac{\partial X^i}{\partial x^m}=y^m\frac{\delta X^i}{\delta x^m}$, we get
\begin{align*}
\nabla_{(X^iG^l_i+y^i\frac{\partial X^l}{\partial x^i})\frac{\partial}{\partial y^l}}\frac{\partial}{\partial x^k}&=(X^iG^l_i+y^i\frac{\partial X^l}{\partial x^i})C^m_{~lk}\frac{\partial}{\partial x^m}\\
&=(y^m\frac{\delta X^l}{\delta x^m}+y^mX^rF^l_{rm})C^m_{~lk}\frac{\partial}{\partial x^m}\\
&=y^i\nabla_iX^lC^m_{~lk}\frac{\partial}{\partial x^m}.
\end{align*}
So in local coordinate, we deduce that 
\begin{align}
L_{\hat X}g(\varrho\hat Y,\varrho\hat Z)=\nabla_iX_j+\nabla_jX_i+2y^m\nabla_mX^kC_{kij}.
\end{align}
By means of the global inner product (\ref{eqn10}), we define the adjoint of this operator.
\begin{lem}
Let $(M,g)$ be a compact Finslerian manifold and $h$ an arbitrary symmetric $2$-form in $S^2\pi_s^*T^*M$. Then the adjoint of Lie derivative of $h$ in local coordinates is given by
\begin{align}
\delta h=-(\nabla^ih_{ik}-h_{kj}\nabla_0C^j+\dot{C}_{kij}h^{ij}+C_{kij}\nabla_oh^{ij})\label{div},
\end{align}
\end{lem}
\begin{proof}
\begin{align*}
\int_{SM}\frac{1}{2}(L_{\hat X}g,h)\eta&=\frac{1}{2}\int_{SM}(\nabla_iX_j+\nabla_jX_i+2y^m\nabla_mX^kC_{ijk})h^{ij}\eta\\
&=\int_{SM}\nabla_iX_jh^{ij}\eta+\int_{SM}y^m\nabla_mX^kC_{kij}h^{ij}\eta\\
&=\int_{SM}(h_{ik}\nabla_0C^i-\nabla^ih_{ij}-(\nabla_0C_{ijk})h^{ij}-C_{ijk}\nabla_0h^{ij})X^k\eta\\
&=-\int_{SM}(\nabla^ih_{ik}-h_{ik}\nabla_0C^i+\dot{C}_{kij}h^{ij}+C_{kij}\nabla_oh^{ij})X^k\eta\\
&=\int_{SM}(X,\delta h)\eta.
\end{align*}
\end{proof}
\begin{thm}
The Berger-Ebin decomposition of $T_g{\cal M}_F\subset S^2\pi^*_sT^*M$ is $T_g{\cal M}_F=\{h|h=L_{\hat X}g\}\oplus S^T$ where $S^T:=\{h|\delta_gh=0\}$.
\end{thm}
\begin{proof}
Define the differential operator $\tau_g$ for every $g\in\cal{M}_F$ by 
\begin{align*}
\tau_g:\Gamma(TM)\rightarrow T_g\cal{M}_F\\
\tau_g(X):=L_{\hat X}g,
\end{align*}
where $\hat{X}$ is the complete lift of $X$. The adjoint of this operator is denoted by $\tau^*$ and defined as follows:
\begin{align*}
\tau^*_g:T_g{\cal M}_F\rightarrow \Gamma (TM)\\
\tau^*_g(h)=-\sharp\delta_gh.
\end{align*}
For an arbitrary vertical 1-form $t$ on $SM$, the symbol of $\tau$ is defined by:
$$\sigma_t(\tau)=t\otimes X_{\sharp}+X_{\sharp}\otimes t,$$
It is injective so the Berger-Ebin decomposition of $T_g\cal{M}_F$ is as follows:
\begin{align}\label{eqn11}
T_g{\cal M}_F=Im\tau_g\oplus ker\tau^*_g,
\end{align}
where $Im\tau_g=\{h|h=L_{\hat X}g\}$ and $ker\tau^*_g=\{h|\delta_gh=0\}$.
\end{proof}
\begin{remark}
By means of decomposition (\ref{eqn11}), every $h$ in $T_g{\cal M}_F$ is decomposed as $h=h_0+L_{\hat{X}}g$. So this decomposition is unique up to the Finslerian Killing vector fields.
\end{remark}
The point-wise conformal deformation of a Finslerian metric $g$ is defined $\tilde{g}(x,y)=f(x)g(x,y)$ where, $f$ is a smooth positive function on $M$, \cite{Kn}. Since there is a one to one correspondence between the space of positive functions and space of exponential functions
by $f\rightarrow e^f$, we can write $\tilde{g}=e^fg$. Let ${\cal P}$ be the product group of positive functions on $M$ that acts on ${\cal M}_F$ as follows:
\begin{align*}
A:{\cal P}\times{\cal M}_F\rightarrow {\cal M}_F\\
A(f,g):=fg,
\end{align*}
This action is free and smooth. The orbit of this action at $g\in{\cal M}_F$ is defined by $A_g=\{fg|f\in{\cal P}\}$ which is a submanifold of ${\cal M}_F$ \cite{S}. The tangent space of this submanifold at $g$ is defined by ${\cal F}g=\{h=kg|k\in C^{\infty}(M)\}$ which is a subbundle of $S^2\pi_s^*T^*M$ at each point $g\in {\cal M}_F$. The orthogonal subspace of ${\cal F}g$ with respect to the global inner product is $S^T:=\{h\in S^2\pi^*_sT^*M|\int_{SM}kgh\eta=0\}=\{h\in S^2\pi^*_sT^*M|tr(h)=0\}$. On the other hand, by means of  the variation of volume forms \cite{A2}, $tr(h)=0$ if and only if $SM$ has constant volume. So the orthogonal space of ${\cal F}g$ is the space of $2$-forms which preserve volume $SM$ through metric variations. Thus, there is a point-wise decomposition like 
\begin{align}
T_g{\cal M}_F={\cal F}g\oplus S^T.\label{1}
\end{align}
Let ${\cal D}$ be the group of infinitesimal diffeomorphism on $M$ and ${\cal P}$ be a $1$-parameter group of positive function on $M$. Put ${\cal C}={\cal D}\times{\cal P}$ which is a semi-direct group with the following action:
\begin{align*}
(\eta_1,f_1).(\eta_2,f_2)=(\eta_1 o \eta_2,f_2(f_1 o \eta_2)),
\end{align*}
This group acts on ${\cal M}_F$ by function $\tilde{A}$ as follows:
\begin{align*}
{\tilde A}:{\cal C}\times{\cal M}_F\rightarrow {\cal M}_F\\
\tilde{A}((\eta, f),g)=f(\tilde{\eta}^*g),
\end{align*}
 The orbit of ${\tilde A}$ passing through $g\in\cal{M}_F$ is
\begin{align*}
{\tilde A}_g:{\cal C}\rightarrow {\cal M}_F\\
{\tilde A}_g(\eta,g)=f(\tilde{\eta}^*g),
\end{align*}
which is a submanifold of ${\cal M}_F$ \cite{S}.
\begin{thm}
The York decomposition of ${\cal B}\subset T_g{\cal M}_F$ is ${\cal B}={\cal F}g\oplus S^{TT}\oplus (S^T\cap Im\tau_g)$, where ${\cal B}$ is defined as the solution space of the system $\frac{\partial h^i_j}{\partial y^k}=0$.
\end{thm}
\begin{proof}
 Define $\tau_g:=d{\tilde A}_g|_{(e,1)}$ as follows:
\begin{align*}
\tau_g:\Gamma(TM)\times C^\infty(M)\rightarrow T_g{\cal M}_F\\
\tau_g(X,k)=L_{\hat X}g+kg,
\end{align*}
The adjoint of $\tau_g$ is denoted by $\tau^*_g$ and defined by:
\begin{align*}
\tau^*_g:T_g{\cal M}_F\rightarrow\Gamma(TM)\times C^\infty(M)\\
h\rightarrow(\sharp div h,tr(h)),
\end{align*}
The condition $\frac{\partial h^i_j}{\partial y^k}=0$ leads to the function $tr(h)$ is just function of $x$. So $\tau^*$ is well-defined.
The kernel of this map is $S^{TT}=\{h\in T_g{\cal M}_F| div h=0, tr(h)=0\}$, and since the symbol of the map $\tau_g$ i.e. $\sigma_t(\tau_g)(X,f)=fg+t\otimes  X_{\sharp}+X_{\sharp}\otimes t$ where, $t$ is an arbitrary vertical $1$-form on $SM$ is injective so the Berger-Ebin decomposition is
\begin{align*}
T_g{\cal M}_F=S^{TT}\oplus Im\tau_g,
\end{align*}
By corresponding this decomposition with point-wise decomposition (\ref{1}), we get
\begin{align}\label{eq9}
T_g{\cal M}_F={\cal F}g\oplus S^{TT}\oplus (S^T\cap Im\tau_g).
\end{align}
\end{proof}
\begin{example}
The subset ${\cal B}$ of $T_g{\cal M}_F$ is nonempty. Let $F(x,y)$ be a Finsler structure which does not reduce to the Riemannian case. Suppose $\tilde{g}=e^{f(t,x)}g(x,y)$ is an arbitrary curve in ${\cal M}_F$. So $h_i^j=\delta^j_i e^{f_0(x)}$ only depends on the variable $x$.
\end{example}
The last term of equation (\ref{eq9}) shows that every $2$-form $h=L_{\hat X}g+fg$ preserves volume of $SM$ that is $tr(h)=0$. So we must have $f=-(2/n)div(\hat{X})$, that is $h$ is in the form $h=L_{\hat X}g-(2/n)div(\hat{X})g$. 
Let $g\in\cal{M}_F$ and $C_g$ be the isotropic group of action $\tilde{A}$, i.e.
$$\{(\eta, f)\in{\cal C}| f\tilde{\eta}^*g=g\},$$
It is clear that $C_g$ is isomorphic to the conformal deformation group, i.e. 
$$\{\eta\in{\cal D}| \tilde{\eta}^*g=fg , \text{for some f} \in{\cal P}\},$$
The Lie algebra of this group is defined by 
$$K_g=\{(X,k)\in\Gamma(TM)\times C^\infty(M)|L_{\hat{X}}g+kg=0\},$$
So it is diffeomorphic with infinitesimal conformal variation
$$\{X\in\Gamma(TM)|L_{\hat{X}}g=\frac{2}{n}div(\hat{X})g\}.$$
According to the above discussion, research works in \cite{A2} and \cite{PB} are restricted to the isotropic group of the Finsler metrics.

\section{Curvature functional on $\cal{M}_F$}
If $H$ is an inner product space with a smooth functional $E:H\rightarrow \mathbb{R}$, the gradient vector field $\nabla E:H\rightarrow H$ is given at each point $u\in H$ by the unique 
vector $\nabla E(u)\in H$ such that for all $u\in H,$ $$(\nabla E(u),V)=dE(u)V$$
\begin{definition}
The gradient flow equation of a functional $E$ is defined as follows:
\begin{align*}
\frac{d}{dt}\varphi_u(t)&=-\nabla E(\varphi_u(t))\\
\varphi_u(0)&=u_0
\end{align*}
where, $\varphi:I\times H\rightarrow H$ is a curve in $H$.
\end{definition}
\begin{definition}
A variation of a Finslerian metric $g_o$ is a $1$-parameter family of metrics $\{g_{t}\}_{t\in I}$, where $g_t=g_o+th$, $g_o\in{\cal M}_F$
and $h\in T_g {\cal M}_F$.
\end{definition}
According to the above definition, the variation of a Finslerian metric is a curve on the manifold ${\cal M}_F$ such that
its tangent vector field is $h:=\partial_t g_t$.
When a Finslerian metric is deformed, then the geometric structures, like nonlinear coefficients, curvature tensors, volume forms and Indicatrix 
will be changed as well. Variations of these objects are calculated in \cite{A2},
\begin{align}
\eta'&=(g^{ij}-\frac{n}{2}u^iu^j)h_{ij}\eta. \label{eq0}\\
V(t)'&=\frac{1}{2}\int_{SM}tr(h)\eta=\frac{n}{2}\int_{SM}t(u,u)\eta\label{eqn 24}\\
G'^i_k&=\frac{1}{2}(\nabla_kh^i_o+\nabla_oh^i_k-\nabla^ih_{ok})-2C^i_{ks}G'^s. \label{eq1}\\
R'^{i}_{~jkl}&=\nabla_k\Lambda^i_{~jl}-\nabla_l\Lambda^i_{~jk}+P^i_{~jlr}\Lambda^r_{~ok}-P^i_{~jkr}\Lambda^r_{ol}+C'^i_{~jr}R^r_{~okl},
\end{align}
where
$$\Lambda^i_{~jk}=\Gamma'^i_{~jk}+C^i_{~jr}\Gamma'^r_{ok}.$$ and 
\begin{align*}
\Gamma'^i_{~jk}&=\frac{1}{2}g^{im}(\nabla_k h_{mj}+\nabla_jh_{mk}-\nabla_mh_{jk})\\
&-(C^i_{~js}G'^s_k+C^i_{~ks}G'^s_j-C_{kjs}G'^s_mg^{im}).
\end{align*}
\begin{align}\label{eqn 25}
\hat{H}'_{jk}&=\tilde{H}_{jk}-\lambda H(u,u)u_ju_k-(n\tau-\phi)u_ju_k
\end{align}
where
\begin{align*}
\tau:=(\nabla^i\nabla_0T_i-\nabla_0T_i\nabla_0T_i)+g^{ij}\partial_j(\nabla_0\nabla_0T_i)\\
\phi:=\frac{1}{2}[\nabla_i\gamma^i-\gamma_i\nabla_0T^i-F^2g^{ij}\partial_i(\psi_j/F)]
\end{align*}
and
$$\gamma_i:=2\lambda\nabla_0T_i-\nabla_i\lambda-T_i\nabla_0\lambda$$
\begin{definition}
A real valued function $E$ on ${\cal M}_F$ is called Finslerian functional if it satisfies the condition $E((d\varphi)^*g)=E(g)$ for every diffeomorphism $\varphi$ on $M$. 
\end{definition}
\begin{example}
The curvature functional (\ref{eqn 19}) is a Finslerian functional. Let $\varphi$ be a diffeomorphism on $M$ and $g$ be a Finslerian metric on manifold $M$ so $(d\varphi)^*g\in{\cal M}_F$. It is easily seen that $\varphi$ is an isometry between two Finslerian manifolds $(M,g)$ and $(M,(d\varphi)^*g)$. Hence $\hat{H}_{(d\varphi)^*g_0}=\hat{H}_{g_0}$, $\eta_{(d\varphi)^* g_0}=\eta_{ g_0}$ and $\tilde{SM}=SM$ and consequently $I((d\varphi)^*g)=I(g)$. So the functional (\ref{eqn 19}) only depends on Finslerian geometric data, and can be viewed as a function on the quotient space ${\cal M}_F/{\cal D}$, where ${\cal D}$ denotes the diffeomorphism group of $M$.
\end{example}
\begin{lem}
The variation of the volume form (\ref{eqn 22}) with respect to the point-wise conformal deformation at $t=0$ is $\eta'=\frac{1}{2}tr_g(h)\eta$.
\end{lem}
\begin{proof}
The point-wise conformal variation of a metric $g$ is ${\tilde g}_{ij}=e^{2f(t,x)}g_{ij}$ so $h_{ij}=\varrho(t,x)g_{ij}$, where $\varrho(t,x)=f'(t,x)e^{f(t,x)}=\frac{1}{n}tr_g(h)$. Substitute this equation in (\ref{eq0}), at $t=0$ we get $\eta'=\frac{1}{2}tr_g(h)\eta$.
\end{proof}
\begin{thm}
Let $(M,g)$ be a closed and connected Finslerian manifold with dim$M\geq 3$. A metric $g_0$ at the critical point$(t=0,g_0=g(0))$ of the functional $I(g_t)$ is a Ricci-directional flat metric. 
\end{thm}
\begin{proof}
Derivative of the functional $I(g_t)$ in an arbitrary direction leads to
\begin{align}
A_{jk}h^{jk}=({\tilde H}_{jk}-\lambda H(u,u)u_ju_k-(n\tau-\phi)u_ju_k-\hat{H}(g_{jk}-\frac{n}{2}u_j u_k))h^{jk}=0.\label{eq2}
\end{align}
Since $g_0$ gives the extremum of $I(g_t)$ and by means of (\ref{eqn 20}), $\hat{H}$ is just a function of variable $x$ at $t=0$. So equation (\ref{eq2}) reduces to
\begin{align}
A_{jk}h^{jk}=({\tilde H}_{jk}-\lambda H(u,u)u_ju_k-(n\tau-\phi)u_ju_k-\frac{1}{2}\hat{H}g_{jk})h^{jk}=0.\label{eq23}
\end{align}
Contracting both sides of $A_{jk}$ by $u^k$ and $u^j$, we have:
\begin{align}
{\tilde H}(u,u)-\lambda H(u,u)-(n\tau-\phi)-\frac{{\hat H}}{2}=0.\label{eq3}
\end{align}
By contraction of $A_{jk}$ by $g^{jk}$, we obtain:
\begin{align}
{\tilde  H}-\lambda H(u,u)-(n\tau-\phi)-\frac{n}{2}{\hat H}=0,\label{eq4}
\end{align}
By subtracting (\ref{eq3}) and (\ref{eq4}), we get:
\begin{align}
\frac{n-1}{2}{\hat H}=-{\tilde H}+{\tilde H}(u,u),\label{eq5}
\end{align}
and 
\begin{align}
\lambda H(u,u)+(n\tau-\phi)=\frac{n}{n-1}{\tilde H}(u,u)-\frac{1}{n-1}{\tilde H}.\label{eq6}
\end{align}
Replacing two last equations in $A_{jk}$ and contracting by $u^j$ and $u^k$ we have:
$$H(u,u)=0.$$
Hence proof is complete.
\end{proof}
\begin{remark}
If in the final step, we contract the equation by $g^{jk}$ instead of $u^j$ and $u^k$ then we obtain $\tilde{H}=0$.
\end{remark}
This functional is not invariant under the rescaling. For eliminating this problem, we use a normal factor $\psi=\psi(t)$, and put $\tilde{g}=\psi(t)g(t)$ such that $\int_{SM}\tilde{\eta}=1$. Therefore, $\eta=\psi^{\frac{-n}{2}}\tilde{\eta}$ and by replacing it in the volume formula, we have $\psi=(V(t))^{\frac{-2}{n}}$. Next, we rewrite the functional $I(g_t)$ with respect to this normalized factor
\begin{align*}
\tilde{I}(g)&=I(\tilde{g}_t)=\int_{SM}(H(\tilde{g})-\lambda H(u,u)(\tilde{g}))\tilde{\eta},\\
&=\int_{SM}\psi^{-1}(H(g)-\lambda H(u,u)(g))\psi^{\frac{n}{2}}\eta,\\
&=\psi^{\frac{n-2}{2}}I(g),\\
&=(V(t))^{\frac{2-n}{n}}I(g).
\end{align*}
\begin{thm}
Let $M$ be a closed and connected Finslerian manifold with dim$M\geq 3$. A metric $g_0$ is a critical point for $\tilde{I}(g_t)$ under all point-wise conformal variations at $t=0$ if and only if the Finslerian manifold is of constant Ricci-directional curvature. 
\end{thm}
\begin{proof}
Derivative of both sides of equation $\tilde{I}(g_t)=(V(t))^{\frac{2-n}{n}}I(g)$ and calculate it at $t=0$:
\begin{align*}
\tilde{I}'(g_t)|_{t=0}&=\frac{2-n}{n}V(t)'|_{t=0}(V(0))^{\frac{2-n}{n}-1}I(g_0)+v(0)^{\frac{2-n}{n}}I'(g_t)|_{t=0},\\
&=V(0)^{\frac{2-n}{n}}\{\frac{2-n}{2n}\frac{I(g_0)}{V(0)}\int_{SM}tr(h)\eta+\int_{SM}A_{ij}h^{ij}\eta\}|_{t=0}.
\end{align*}
Put $Ave:=\frac{I(g_0)}{V(0)}$ which is a constant value. Restrict to the point-wise conformal deformation, we get:
\begin{align}
0&=\tilde{I}'(g_t)|_{t=0}\nonumber\\
&=V(0)^{\frac{2-n}{n}}\int_{SM}(\frac{2-n}{n}Ave+A_{ij}g^{ij})\frac{tr_g(h)}{n}\eta,
\end{align}
Since $h$ is an arbitrary $2$-form in ${\cal F}g$, we have:
\begin{align}
0=\frac{2-n}{2}Ave+A_{ij}g^{ij}=\frac{2-n}{2}Ave-\tilde{H}+\lambda H(u,u)+(n\tau-\phi)+\frac{n}{2}\hat{H},\label{eq7}
\end{align}
Substituting (\ref{eq5}) and (\ref{eq6}) into (\ref{eq7}) and using (\ref{eqn 20}) we obtain:
\begin{align*}
H(u,u)=-\frac{(n-2)}{4n}Ave.
\end{align*}
\end{proof}
\begin{remark}
In the set of stationary points of the curvature functional $I(g_t)$, based on the constant Indicatrix volume, we have (\ref{eqn 20}). Hence $\tilde{H}$ is constant, as well.
\end{remark}
\begin{mycor}
The unnormalized gradient flow of $\tilde{I}(g_t)$ with respect to the subspace ${\cal F}g$ of $T_g{\cal M}_F$ is
\begin{align}
\frac{\partial}{\partial t}g_{ij}(t,z)=-H_t(u,u)g_{ij}(t,z).\label{eqn 24}
\end{align}
and it is a strictly parabolic equation.
\end{mycor}
\begin{proof}
Derivative of the functional $I(g_t)$ is given by 
$$I'(g_t)|_{t=0}=\int_{SM}H(u,u)tr_g(h)\eta=\int_{SM}H(u,u)g_{jk}h^{jk}\eta=0$$ 
so its Euler-Lagrange equation is given by $H(u,u)g_{ij}=0$. It follows that its associated gradient flow is $\frac{\partial}{\partial t}g_{ij}=-H_t(u,u)g_{ij}(t)$.
The linearization of this equation is
\begin{align*}
D[H(u,u)g_{ij}]:C^\infty(S^2(T^*TM))\rightarrow C^\infty(S^2(T^*TM)),\\
D[H(u,u)g_{ij}](\frac{\partial g_{ij}}{\partial t})=D[H(u,u)](h_{ij})=\frac{\partial}{\partial t}\tilde{H}(u,u).
\end{align*}
So we have 
\begin{align*}
\frac{\partial}{\partial t}\tilde{H}(u,u)&=\frac{\partial}{\partial t}(\tilde{H}_{ij}\frac{y^i}{\tilde{F}}\frac{y^j}{\tilde{F}}),\\
&=\tilde{F}^{-2}(\nabla_s\nabla_0h_0^s-\frac{1}{2}\nabla_s\nabla^sh_{00}-\nabla_0\nabla_0h^s_s+{\text{lower order terms}})\tilde{g}_{ij}.
\end{align*}
The total symbol of the Ricci directional curvature $H(u,u)$ is 
\begin{align*}
\sigma[H(u,u)g_{ij}](\xi)(h_{ij})=\tilde{F}^{-2}(h^i_0\xi_i\xi_sy^s-\frac{1}{2}\xi_i\xi^ih_{00}-y^s\xi_sy^l\xi_lh_i^i)g_{ij}+\text{lower order terms}.
\end{align*}
So the principal symbol of the tensor $H(u,u)g_{ij}$ is 
\begin{align*}
\hat{\sigma}[H(u,u)g_{ij}](\xi)(h_{ij})=\tilde{F}^{-2}(h^i_0\xi_i\xi_sy^s-\frac{1}{2}\xi_i\xi^ih_{00}-y^s\xi_sy^l\xi_lh_i^i)\tilde{g}_{ij}.
\end{align*}
Put $\xi_1=1$ and $\xi_j=0$ for all $j\neq 1$. To evaluate the principal symbol of this equation, we take an orthonormal frame $(e_i)$ at $x\in M$ such that $u^n=\frac{y^n}{F}=1$ and $u^\alpha=0$ for all $\alpha\neq n$, it is clear that
\begin{align*}
\hat{\sigma}[H(u,u)g_{ij}](x,y)(\xi ,h_{ij})&=(h^i_l\xi_i\xi_su^su^l-\frac{1}{2}h(u,u)-u^s\xi_su^l\xi_ltr(h))\tilde{g}_{ij},\\
&=-\frac{1}{2}h(u,u)\tilde{g}_{ij}=-\partial_tlog \tilde{F}\tilde{g}_{ij}.
\end{align*}
Hence it is a strictly parabolic equation.
\end{proof}
The normalized gradient flow of functional $I(g_t)$ with restricted to the point-wise conformal deformation is $\frac{\partial}{\partial t}g_{ij}(z,t)=-(H_t(u,u)-c(t))g_{ij}(t,z)$, where $c(t)$ is a constant value at each $t$ and is defined by $c(t)=\frac{\int_{SM}\hat{H}_t\eta_t}{\int_{SM}\eta_t}$.
\begin{remark}
According to the above discussion, we can define (\ref{eqn 24}) by means of $\tilde{H}$, i.e.
$$\partial_tg_{ij}(t,z)=-\tilde{H}_tg_{ij}(t,z).$$
\end{remark}
According to the Akbar-Zadeh's calculations, the Euler-Lagrange equation of functional (\ref{eqn 19}) for an arbitrary direction is $-\tilde{H}_{ij}+c(x)g_{ij}=0.$ So its associated gradient flow is 
\begin{align}
\frac{\partial g_{ij}}{\partial t}(t,z)=-\tilde{H}_{ij}(t,z)+H_t(u,u)g_{ij}(t,z)=-\nabla I(g_t)\label{eqn 26}
\end{align} 
Consider the linearization of this equation. Since $\tilde{H_{ij}}'=\frac{1}{2}\frac{\partial^2}{\partial y^i\partial
y^j}(H'_{kr}y^ky^r)$ and $H'_{ks}y^ky^s=2\nabla_rG'^r-\nabla_0G'^r_r+2\nabla_0T_rG'^r$, cf. \cite{A2} and use (\ref{eq1}), we deduce that the first term of (\ref{eqn 26}) is of order $4$ in term of $h$. Note that similar to the Riemannian case, (\ref{eqn 26}) has not any solution since the second term of (\ref{eqn 26}) is a backward equation. Bao considered the first term of (\ref{eqn 26}) as the Ricci flow on Finsler manifolds, i.e. $\partial_tg_{ij}=-\tilde{H}_{ij}(t)$. Through this work, we derive the second term of (\ref{eqn 26}), i.e. $\partial_tg_{ij}=-H_t(u,u)g_{ij}(t)$ as a gradient flow in the special direction of variations. Both of these flows have scalar form $\partial_tlogF_t=-H_t(u,u)$. So we prefer to use the tensor forms of flows for our later studying and use manifold ${\cal M}_F$ to garantees their solutions come from Finsler structure.
\bibliographystyle{plain}

\bibliography{paper}

\end{document}